\documentclass[reqno]{amsart}

\usepackage{amsmath}
\usepackage{amssymb}
\usepackage{amsfonts}
\usepackage{amsthm}

\usepackage{subcaption}

\usepackage{hyperref}

\usepackage{graphicx}
\usepackage{tikz}

\usepackage{algorithm}
\usepackage{algpseudocode}

\newtheorem*{thm}{Theorem}
\newtheorem{lemma}{Lemma}

\begin{document}

\title[]{Juniper Green and the \\Gallai-Edmonds Decomposition}

\author[]{Tony Zeng}
\address{Department of Mathematics, University of Washington, Seattle, WA 98195, USA}
\email{txz@uw.edu}

\begin{abstract}
    Juniper Green is a simple combinatorial game invented by Rob Porteous and popularized by Ian Stewart. It was originally designed to familiarize school children with the concepts of multiplication and division. We analyze this elementary game through a completely different lens and show that it recovers the Gallai-Edmonds decomposition of the divisibility graph on the vertex set $V= \left\{1,2,\dots, n\right\}$. This characterizes the winning moves of the game; as a byproduct, we show that this decomposition seems to have many interesting and curious patterns that are currently unexplained.
\end{abstract}

\maketitle

\section{Introduction}
\subsection{History.} Juniper Green is a simple combinatorial game that was originally invented by Rob Porteous to teach children about multiplication and division (Juniper Green is the name of the school where Porteous taught). Porteous explained it to his father, the mathematician Ian Porteous, who in turn explained it to Ian Stewart, who then popularized it in a 1997 article \cite{stewart_art} in \textit{Scientific American} and further described in his 2010 book \cite{stewart}. The history may be more complicated, Paul Blatz \cite{stewart1997lore} recalls that the game was discussed in a number theory
course given at Princeton University by Eugene Wigner in the late 1930s. The game has since also been discussed in the pedagogical literature \cite{billstein2013problem, stallings1999juniper, uugurel2008matematik}.

\subsection{Rules.}
The game involves two players and is played on \([n] = \left\{1,2,\dots, n\right\}\) subject to the following rules. 
\begin{enumerate}
    \item Players take turns selecting numbers from \([n]\) such that each selected number must be a factor or multiple of the previously selected number.
    \item Once a number has been used, it cannot be selected again.
    \item A player loses if they have no legal moves.
\end{enumerate}

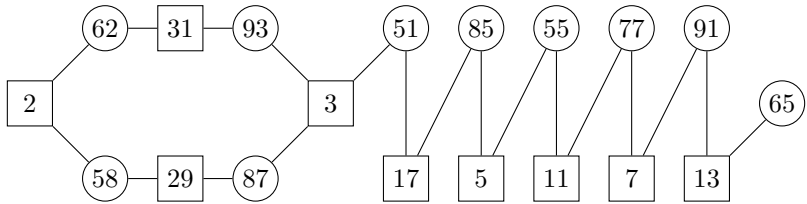
\begin{figure}[h!]
    \centering
    \begin{tikzpicture}
        \node[draw, inner sep=0pt, minimum size=0.6cm] (2) at (0,0) {\(2\)};
        \node[draw, circle, inner sep=0pt, minimum size=0.6cm] (62) at (1, 1) {\(62\)};
        \node[draw, circle, inner sep=0pt, minimum size=0.6cm] (58) at (1, -1) {\(58\)};
        \node[draw, inner sep=0pt, minimum size=0.6cm] (31) at (2, 1) {\(31\)};
        \node[draw, inner sep=0pt, minimum size=0.6cm] (29) at (2, -1) {\(29\)};
        \node[draw, circle, inner sep=0pt, minimum size=0.6cm] (93) at (3, 1) {\(93\)};
        \node[draw, circle, inner sep=0pt, minimum size=0.6cm] (87) at (3, -1) {\(87\)};
        \node[draw, inner sep=0pt, minimum size=0.6cm] (3) at (4, 0) {\(3\)};
        \node[draw, circle, inner sep=0pt, minimum size=0.6cm] (51) at (5, 1) {\(51\)};
        \node[draw, inner sep=0pt, minimum size=0.6cm] (17) at (5, -1) {\(17\)};
        \node[draw, circle, inner sep=0pt, minimum size=0.6cm] (85) at (6, 1) {\(85\)};
        \node[draw, inner sep=0pt, minimum size=0.6cm] (5) at (6, -1) {\(5\)};
        \node[draw, circle, inner sep=0pt, minimum size=0.6cm] (55) at (7, 1) {\(55\)};
        \node[draw, inner sep=0pt, minimum size=0.6cm] (11) at (7, -1) {\(11\)};
        \node[draw, circle, inner sep=0pt, minimum size=0.6cm] (77) at (8, 1) {\(77\)};
        \node[draw, inner sep=0pt, minimum size=0.6cm] (7) at (8, -1) {\(7\)};
        \node[draw, circle, inner sep=0pt, minimum size=0.6cm] (91) at (9, 1) {\(91\)};
        \node[draw, inner sep=0pt, minimum size=0.6cm] (13) at (9, -1) {\(13\)};
        \node[draw, circle, inner sep=0pt, minimum size=0.6cm] (65) at (10, 0) {\(65\)};
        \draw (65) -- (13) -- (91) -- (7) -- (77) -- (11) -- (55) -- (5) -- (85) -- (17) -- (51) -- (3) -- (87) -- (29) -- (58) -- (2) -- (62) -- (31) -- (93) -- (3);
    \end{tikzpicture}
    \caption{An induced subgraph of the divisibility graph for \(n = 100\) on the vertices relevant to a particular winning strategy.}
    \label{fig:strategy}
\end{figure}

One can check that the first player has a straightforward winning strategy by picking a prime \(p\) greater than \(n / 2\). This forces a response of \(1\), allowing player one to end the game by picking a second prime \(q\) also greater than \(n / 2\). This requires only the existence of at least \(2\) distinct primes \(p, q \in (n / 2, n)\). For sufficiently large \(n\), this is guaranteed to be the case. Because of this ``boring'' strategy, it is standard to introduce the additional rule.
\begin{enumerate}
    \item[(4)] The first move must be even\footnote{One might reasonably ask why this modification is not to instead restrict the first move to a composite. We asked ourselves the same thing and have no satisfying answer.}.
\end{enumerate}

Stewart described this game in \cite{stewart_art, stewart} and for \(n = 100\), describes how the first player can win with \(58\) or \(62\) as the opening move.  
A complete description of the possible sequence of moves in that case is shown in Figure \ref{fig:strategy}. It is unclear whether this particular strategy for $n=100$ can generalize. Stewart's article finishes with the following question.
\begin{quote}
What about [...] completely general $n$?  Can anyone find patterns?  Or solve the whole thing? (Ian Stewart \cite{stewart_art})
\end{quote}

\begin{figure}[h!]
	\centering
    \includegraphics[scale=0.2]{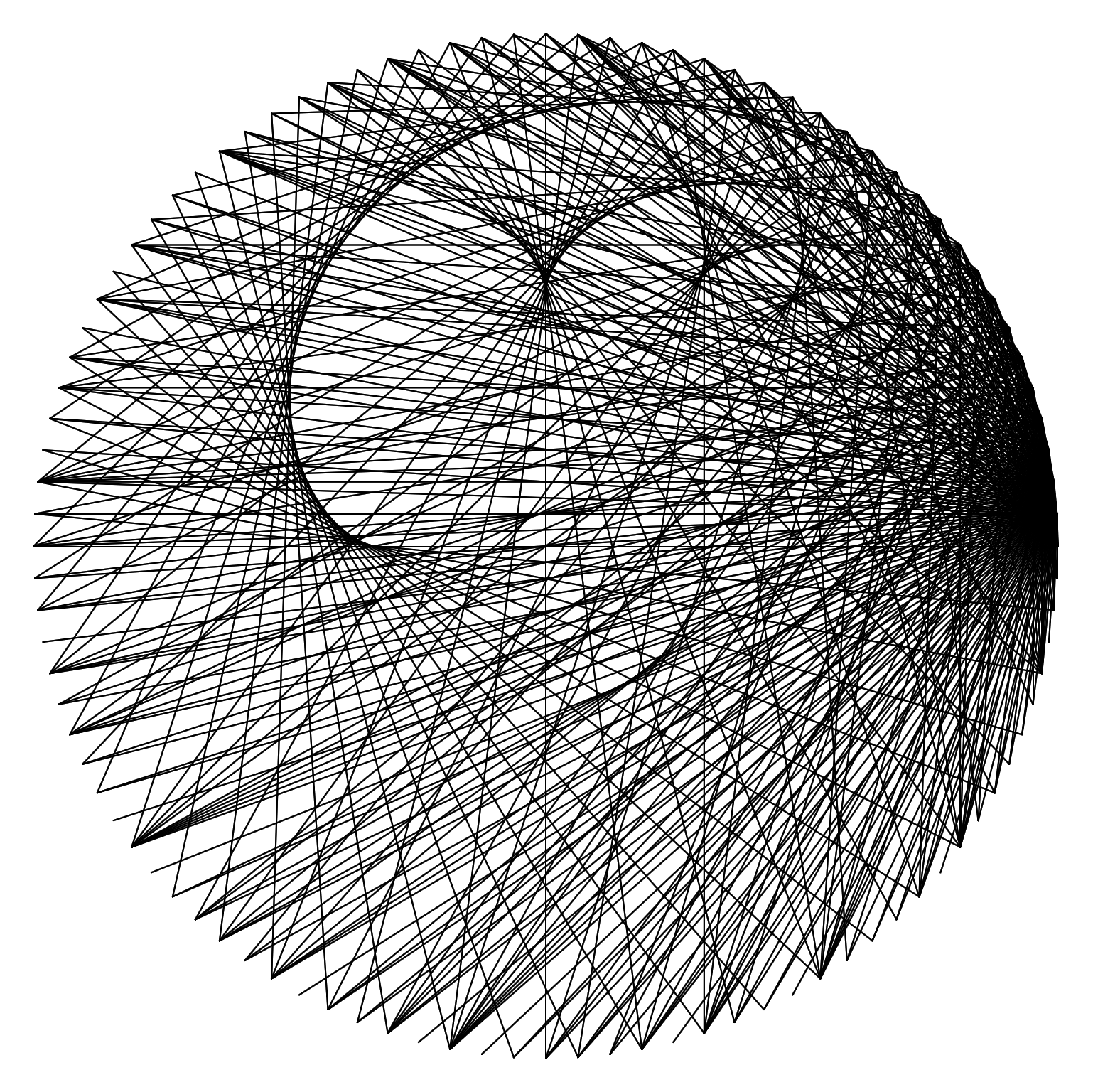} 
    \caption{The divisibility graph on \([n]\) for \(n = 100\), where the vertices are arranged in a circle. In particular, vertex \(i\) is placed at \(\left(\cos(2\pi i / n), \sin(2\pi i / n)\right) \in \mathbb{R}^2\). The curves that appear as envelopes of these edges are epicycloids.}
    \label{fig:div_graph}
\end{figure}

\subsection{Lemoine's Theorem} Stewart's question was completely answered by Julien Lemoine in a 2022 paper.

\begin{thm}[Lemoine \cite{lemoine}] For any $n \geq 120$, the first player has a winning strategy.
\end{thm}

The argument relies on the fact that there always exist \(3\) primes in \([n / 4, n / 3]\) for \(n\) sufficiently large, a similar strategy for the first player can be generalized for arbitrarily large \(n\), leaving only finitely many cases to directly analyze. Lemoine furthermore directly analyzes the finitely many remaining cases in $n \leq 120$ and tabulates which player can force a win by an exhaustive analysis.

\section{Result}
We were motivated by Stewart's question whether more can be said: is it possible to understand the structure of winning moves for each $n$? In this sense, we refer to Juniper Green as the game subject to rules (1) -- (3) and not (4) with the understanding that starting with a large prime number is a winning move for player one (which was the reason for rule (4) in the first place). Our main result can be informally summarized by saying that winning moves in Juniper Green correspond to a specific set in the Gallai-Edmonds decomposition of the divisibility graph on \([n]\) (see Figure \ref{fig:div_graph}). We recall that the divisibility graph \(G_n\) is a graph with vertices $V = \left\{1,2,\dots, n\right\}$, where an edge exists between \(i, j\) when either \(i | j\) or \(j | i\). 

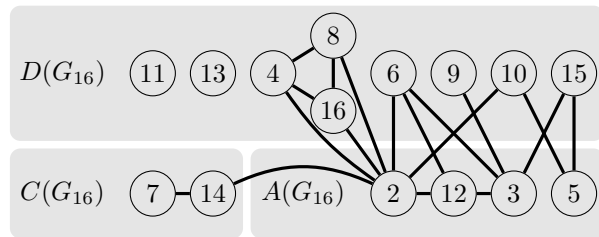
\begin{figure}[h!]
    \centering
    \begin{tikzpicture}[
        every node/.style={draw, circle, inner sep=0pt, minimum size=0.6cm},
        edge/.style={very thick, black}
    ]
        \fill[rounded corners, gray, opacity=0.2] (0.2, 2.9) rectangle (-7.7, 1.1);
        \fill[rounded corners, gray, opacity=0.2] (0.2, -0.2) rectangle (-4.5, 1.0);
        \fill[rounded corners, gray, opacity=0.2] (-4.6, -0.2) rectangle (-7.7, 1.0);

        \node[draw=none] at (-7.0, 2) {\(D(G_{16})\)};
        \node[draw=none] at (-3.8, 0.4) {\(A(G_{16})\)};
        \node[draw=none] at (-7.0, 0.4) {\(C(G_{16})\)};

        \node (4)  at (-4.2, 2) {\(4\)};
        \node (8)  at (-3.4, 2.5) {\(8\)};
        \node (16) at (-3.4, 1.5) {\(16\)};
        \node (6)  at (-2.6, 2) {\(6\)};
        \node (9)  at (-1.8, 2) {\(9\)};
        \node (10) at (-1.0, 2) {\(10\)};
        \node (15) at (-0.2, 2) {\(15\)};
        \node (11) at (-5.8, 2) {\(11\)};
        \node (13) at (-5.0, 2) {\(13\)};

        \node (2)  at (-2.6, 0.4) {\(2\)};
        \node (12) at (-1.8, 0.4) {\(12\)};
        \node (3)  at (-1.0, 0.4) {\(3\)};
        \node (5)  at (-0.2, 0.4) {\(5\)};

        \node (7)  at (-5.8, 0.4) {\(7\)};
        \node (14) at (-5.0, 0.4) {\(14\)};

        \draw[edge] (7) -- (14);
        \draw[edge] (14) edge[bend left=25] (2);
        \draw[edge] (2) edge[bend left=10] (4);
        \draw[edge] (4) -- (8) -- (16) -- (4);
        \draw[edge] (12) -- (6) -- (2);
        \draw[edge] (8) edge[bend left=0] (2);
        \draw[edge] (16) -- (2);
        \draw[edge] (9) -- (3) -- (6);
        \draw[edge] (2) -- (10) -- (5) -- (15) -- (3);
        \draw[edge] (2) -- (12);
        \draw[edge] (3) -- (12);
    \end{tikzpicture}
    \caption{The Gallai-Edmonds Decomposition of the divisibility graph \(G_{16}\) with \(1\) (\(1 \in A(G_{16})\)) omitted to reduce clutter. For $n$ sufficiently large, \(1\) will always be a member of \(A(G_n)\).}
\end{figure}

%
\subsection{The Gallai-Edmonds decomposition} We first describe the main ingredient which was discovered, independently, by Tibor Gallai \cite{gallai1963kritische, gallai1964maximale} and Jack Edmonds \cite{edmonds1965paths}. Given a graph $G=(V,E)$, the Gallai-Edmonds decomposition is a partition of the vertex set into three disjoint sets 
$$ V = D(G) \cup A(G) \cup C(G), $$
as follows: a vertex is said to either be \textit{essential} (if it is covered by every maximum matching in $G$
) or \textit{inessential}. The decomposition is now as follows:
\begin{enumerate}
    \item \(D(G) \subseteq V\) consists of all \textit{inessential} vertices.
  \item \(A(G) \subseteq V\) contains all essential vertices with an inessential neighbor.
        \item \(C(G) \subseteq V\) contains all remaining essential vertices.
\end{enumerate}

\subsection{Main Result.}  
We can now state our main result.

\begin{thm}[Main Result] The winning moves for player one in Juniper Green on \([n]\) are precisely the vertices in $D(G_n)$ of the Gallai-Edmonds decomposition of the divisibility graph \(G_n\). 
\end{thm}

The result follows from Berge's Theorem \cite{berge1957two, berge1996combinatorial}, see \S 3 for the details and a proof. As is often the case, it is difficult to answer a question without raising a new question; here, the new question evidently becomes
\begin{quote}
    What can be said about the Gallai-Edmonds decomposition of the divisibility graph on \([n]\)?
\end{quote}

As shown in Figure \ref{fig:dac_ratios}, it appears nontrivial to say very much at all. The result of Lemoine implies that $|D(G_n) \cap 2\mathbb{N}| \geq 1$ once $n \geq 120$. Computations using the blossom algorithm \cite{Jungnickel2013} to construct the decomposition suggest a most curious picture. It appears as if $|A(G_n)|$ is constantly small while $|C(G_n)|$ and $|D(G_n)|$ seem to dramatically alternate in size: these transitions seem to occur whenever $n$ has many small prime factors (making it a particularly centrally vertex in the divisibility graph). This makes intuitive sense as illustrated with the following example. \(G_{101}\) does not differ much from \(G_{100}\), as only one new vertex and one new edge are added; however, \(G_{100}\) differs very much from \(G_{99}\), as from \(99\) to \(100\) adds many edges to the new vertex. That said, it remains unclear how one can make this precise.
From Figure \ref{fig:dac_membership}, there are a number of observations that can be made. It appears that the majority of the elements of \(A(G_n)\) are less than \(n / 3\), and there are no members of \(A(G_n)\) between \(n / 3\) and \(n / 2\). The interval \([n / 2, 2n / 3]\) seems to have the highest density of elements of \(D(G_n)\). Vertical bands that seem to correspond to primes appear to belong either to \(D(G_n)\) or \(C(G_n)\) depending on \(n\). Horizontal bands seem to indicate some kind of persistence a particular number has in terms of which part of the Gallai-Edmonds Decomposition it belongs to. It is easy to make further guesses, but it seems difficult to prove anything along these lines rigorously.

\begin{figure}[h!]
    \centering
    \begin{tikzpicture}
        \node at (0, 0) {\includegraphics[scale=0.6]{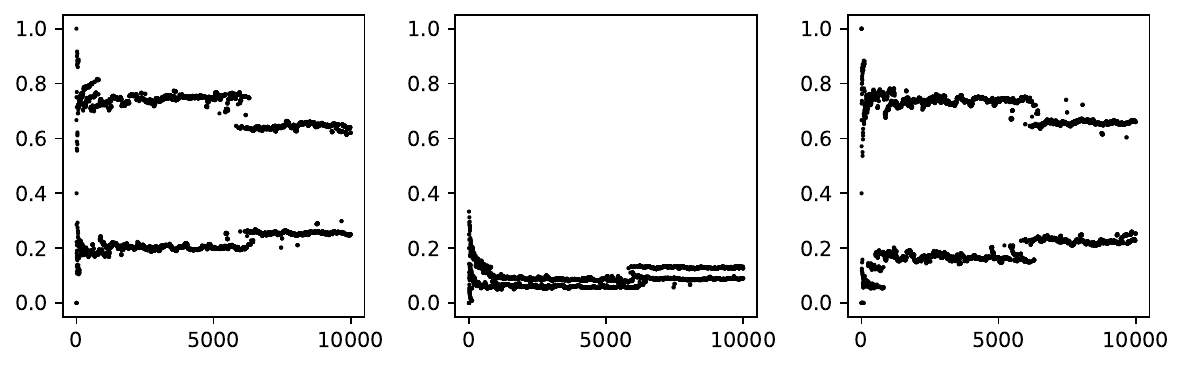}};
        \node at (-3.9, -2) {\(|D(G_n)|/n\)};
        \node at (0.1, -2) {\(|A(G_n)|/n\)};
        \node at (4.1, -2) {\(|C(G_n)|/n\)};
    \end{tikzpicture}
    \caption{Proportions of \(D(G_n)\), \(A(G_n)\), and \(C(G_n)\).}
    \label{fig:dac_ratios}
\end{figure}
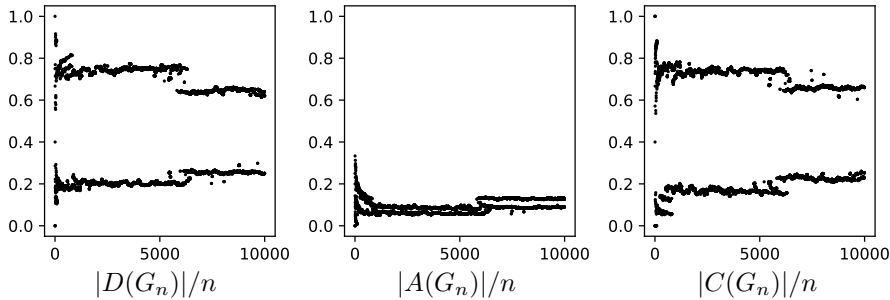

\section{Proof of the Main Result}

Juniper Green is a specific instance of the `snake in the box' game on an undirected graph as described in \cite{berge1996combinatorial}. Snake in the box is a two player combinatorial game subject to the following rules. 
\begin{enumerate}
    \item Players take turns selecting vertices of a graph \(G\) such that each selected vertex must be a neighbor of the previous vertex.
    \item Once a vertex has been used, it cannot be selected again.
    \item A player loses if they have no legal moves.
\end{enumerate}
Notice that indeed, the above rules are entirely analogous to the rules described in \S 1. 
We can therefore understand Juniper Green as snake in the box played on the graph with vertex set \([n]\) and edge set \(\{(i, j) : i|j\}\).
In other words, Juniper Green is snake in the box played on the divisibility graph \(G_n\). 

Recall that a matching of a graph is a set of edges, no two of which share an edge. A maximum matching is simply a maximum cardinality matching. Note that this is not to be confused with a maximal matching, which is a matching \(M\) such that \(M \cup \{e\}\) is not a matching for all \(e \in E \setminus M\). Using maximum matchings, Berge provides a necessary and sufficient condition for the first player to have a winning strategy for snake in the box.

\begin{thm}[Berge, \cite{berge1996combinatorial}]
    For snake in the box on a graph \(G\), the first player has a winning strategy if and only if a maximum matching of \(G\) does not cover every vertex.
\end{thm}

\begin{figure}[h!]
    \centering
    \includegraphics[scale=0.4]{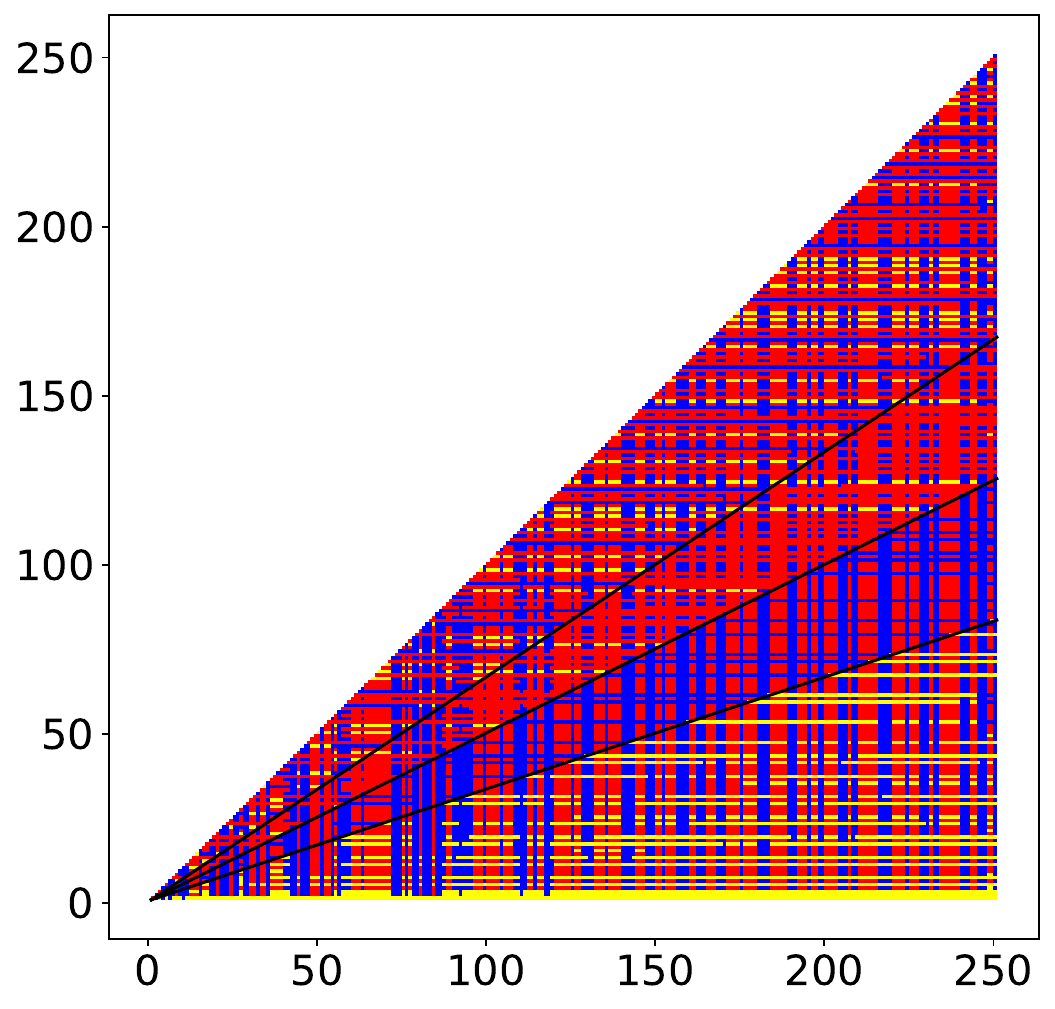}
    \caption{The elements of \([n]\) plotted against \(n\) colored by whether they belong to \(D(G_n)\) (red), \(A(G_n)\) (yellow), or \(C(G_n)\) (blue). The black lines plotted are \(n / 3\), \(n / 2\), and \(2n / 3\).}
    \label{fig:dac_membership}
\end{figure}

The proof of this theorem is quite elegant, using Berge's alternating chain lemma. The proof is important for our purposes as it provides as a corollary an exact set of moves the first player can make to force a win. 

\begin{lemma}[Alternating Chain Lemma]
    A matching \(M\) of a graph \(G\) is a maximum matching if and only if there are no two vertices uncovered by \(M\) that are connected by an alternating chain, i.e. a chain whose edges alternate between belonging to \(M\) and \(E(G) \setminus M\).
\end{lemma}

Interestingly, the alternating chain lemma is sometimes called the alternating chain theorem or Berge's theorem (not to be confused with the above theorem). Its proof is slightly more involved, so for the interested reader, we refer to \cite{berge1996combinatorial}. We now give Berge's proof characterizing which player wins snake in the box.

\begin{proof}[Proof of Theorem]
    Let \(M\) be a maximum matching of \(G\). If \(M\) does not cover a vertex \(v_1\), the first player can select \(v_1\). The second player's response must be covered by some edge \(e = (v_2, v_3) \in M\), as otherwise, \(M' = M \cup e\) would yield a larger matching. Suppose the second player chooses \(v_2\). The first player now can choose the other vertex of \(e\), namely \(v_3\). From this point forward, the second player may not select a vertex uncovered by \(M\), as this would violate the alternating chain lemma. 
    Since the first player opened with a vertex not covered by \(M\), the second player must always select a vertex \(u\) covered by \(e' \in M\) such that if \(e' = (u, u')\), \(u'\) has not yet been selected. This means that regardless of what the second player chooses, the first player will have an available move. Therefore, the first player cannot lose, i.e. player one wins.

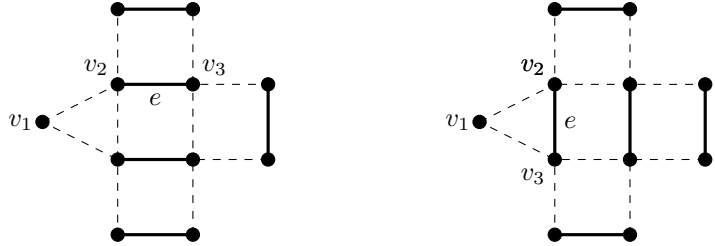
\begin{figure}[h!]
    \centering
    \begin{subfigure}{0.45\textwidth}
        \centering
        \begin{tikzpicture}
            \filldraw (-1, 0) circle (2.5pt) node[anchor=east]{\(v_1\)};
            \filldraw (0, 0.5) circle (2.5pt) node[anchor=south east]{\(v_2\)};
            \filldraw (0, 1.5) circle (2.5pt);
            \filldraw (1, 0.5) circle (2.5pt) node[anchor=south west]{\(v_3\)};
            \node[anchor=north] at (0.5, 0.5) {\(e\)};
            \filldraw (1, 1.5) circle (2.5pt);
            \filldraw (2, 0.5) circle (2.5pt);
            \filldraw (0, -0.5) circle (2.5pt);
            \filldraw (0, -1.5) circle (2.5pt);
            \filldraw (1, -0.5) circle (2.5pt);
            \filldraw (1, -1.5) circle (2.5pt);
            \filldraw (2, -0.5) circle (2.5pt);
            \draw[dashed] (-1, 0) -- (0, 0.5) -- (0, 1.5) -- (1, 1.5) -- (1, 0.5) -- (2, 0.5) -- (2, -0.5) -- (1, -0.5) -- (1, -1.5) -- (0, -1.5) -- (0, 0.5) -- (1, 0.5) -- (1, -0.5) -- (0, -0.5) -- (-1, 0);
            \draw[very thick] (0, 1.5) -- (1, 1.5);
            \draw[very thick] (0, 0.5) -- (1, 0.5);
            \draw[very thick] (0, -0.5) -- (1, -0.5);
            \draw[very thick] (0, -1.5) -- (1, -1.5);
            \draw[very thick] (2, 0.5) -- (2, -0.5);
        \end{tikzpicture}
    \end{subfigure}
    \begin{subfigure}{0.45\textwidth}
        \centering
        \begin{tikzpicture}
            \filldraw (-1, 0) circle (2.5pt) node[anchor=east]{\(v_1\)};
            \filldraw (0, 0.5) circle (2.5pt) node[anchor=south east]{\(v_2\)} node[anchor=south east]{\(v_2\)};
            \filldraw (0, 1.5) circle (2.5pt);
            \filldraw (1, 0.5) circle (2.5pt);
            \filldraw (1, 1.5) circle (2.5pt);
            \filldraw (2, 0.5) circle (2.5pt);
            \filldraw (0, -0.5) circle (2.5pt) node[anchor=north east]{\(v_3\)};
            \node[anchor=west] at (0, 0) {\(e\)};
            \filldraw (0, -1.5) circle (2.5pt);
            \filldraw (1, -0.5) circle (2.5pt);
            \filldraw (1, -1.5) circle (2.5pt);
            \filldraw (2, -0.5) circle (2.5pt);
            \draw[dashed] (-1, 0) -- (0, 0.5) -- (0, 1.5) -- (1, 1.5) -- (1, 0.5) -- (2, 0.5) -- (2, -0.5) -- (1, -0.5) -- (1, -1.5) -- (0, -1.5) -- (0, 0.5) -- (1, 0.5) -- (1, -0.5) -- (0, -0.5) -- (-1, 0);
            \draw[very thick] (0, 1.5) -- (1, 1.5);
            \draw[very thick] (0, 0.5) -- (0, -0.5);
            \draw[very thick] (1, 0.5) -- (1, -0.5);
            \draw[very thick] (0, -1.5) -- (1, -1.5);
            \draw[very thick] (2, 0.5) -- (2, -0.5);
        \end{tikzpicture}
    \end{subfigure}
    \caption{If \(M\) is a maximum matching of \(G\) and \(e \in M\), then any neighbor \(u\) of \(v_3\) must also be covered by some edge of \(M\), as otherwise, replacing \(e\) with \((v_1, v_2)\) and \((v_3, u)\) would yield a matching with higher cardinality.}
    \label{fig:acl}
\end{figure}

    If \(M\) covers every vertex of \(G\), the second player has the following winning strategy. The first player chooses some vertex \(u\), allowing the second player to respond with \(v\), where \((u, v) \in M\). Since every vertex of \(G\) is covered by some edge of \(M\), we see that the first player will always select some vertex \(u'\) such that if \(e' = (u', v') \in M\) covers \(u'\), \(v'\) is available for the second player to choose. 

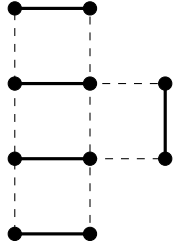
\begin{figure}[h!]
    \centering
    \begin{tikzpicture}
        \filldraw (0, 0.5) circle (2.5pt);
        \filldraw (0, 1.5) circle (2.5pt);
        \filldraw (1, 0.5) circle (2.5pt);
        \filldraw (1, 1.5) circle (2.5pt);
        \filldraw (2, 0.5) circle (2.5pt);
        \filldraw (0, -0.5) circle (2.5pt);
        \filldraw (0, -1.5) circle (2.5pt);
        \filldraw (1, -0.5) circle (2.5pt);
        \filldraw (1, -1.5) circle (2.5pt);
        \filldraw (2, -0.5) circle (2.5pt);
        \draw[dashed] (0, 0.5) -- (0, 1.5) -- (1, 1.5) -- (1, 0.5) -- (2, 0.5) -- (2, -0.5) -- (1, -0.5) -- (1, -1.5) -- (0, -1.5) -- (0, 0.5) -- (1, 0.5) -- (1, -0.5) -- (0, -0.5);
        \draw[very thick] (0, 1.5) -- (1, 1.5);
        \draw[very thick] (0, 0.5) -- (1, 0.5);
        \draw[very thick] (0, -0.5) -- (1, -0.5);
        \draw[very thick] (0, -1.5) -- (1, -1.5);
        \draw[very thick] (2, 0.5) -- (2, -0.5);
    \end{tikzpicture}
    \caption{Whatever the first player selects, the second player can always respond by following the edges of a maximum matching \(M\).}
\end{figure}

Therefore, the second player always has an available move by following \(M\), so they cannot lose, i.e. player two wins.
\end{proof}

Berge's proof gives us all the relevant tools to prove our result.

\begin{proof}[Proof of Main Result]
    Berge's proof tells us that the set of opening moves that the first player can take and win is precisely the set of inessential vertices, those elements \(v \in V\) such that there exists a maximum matching \(M\) that does not cover \(v\). Thus, since Juniper Green is snake in the box on \(G_n\), we conclude that the set of winning openings for the first player is \(D(G_n)\).
\end{proof}

\bibliographystyle{plain}
\bibliography{refs}

\end{document}